\documentclass[11pt, leqno]{amsart}

\usepackage{ifpdf}
\ifpdf
\usepackage[pdftex,plainpages=false,pdfpagelabels,pdfstartview=FitH]{hyperref}
\else
	\usepackage[hypertex]{hyperref}
\fi

\theoremstyle{plain}
\newtheorem{theorem}{Theorem}[section]
\newtheorem{proposition}[theorem]{Proposition}
\newtheorem{lemma}[theorem]{Lemma}
\newtheorem{corollary}[theorem]{Corollary}

\theoremstyle{definition}

\newtheorem{definition}[theorem]{Definition}
\newtheorem{example}[theorem]{Example}

\numberwithin{equation}{section}

\def\D{\mathrm{d}}
\def\vu{{\vec{u}}}
\def\vc{{\vec{c}}}
\def\vb{{\vec{b}}}
\def\vv{{\vec{v}}}
\def\ve{{\vec{e}}}
\def\vC{{\vec{C}}}

\def\I{\mathrm{I}}
\def\II{\mathrm{II}}
\def\diag{\operatorname{diag}}
\def\sgn{\operatorname{sgn}}
\def\ad{\operatorname{ad}}
\def\Aut{\operatorname{Aut}}
\def\rk{\operatorname{rk}}

\def\R{\mathbb{R}}
\def\C{\mathbb{C}}
\def\fa{\mathfrak{a}}
\def\fk{\mathfrak{k}}
\def\calM{\mathcal{M}}
\def\fp{\mathfrak{p}}
\def\fu{\mathfrak{u}}
\def\fo{\mathfrak{o}}
\def\rO{\mathrm{O}}
\def\rSO{\mathrm{SO}}
\def\rSU{\mathrm{SU}}

\def\onk{\frac{\rO(2n-k, k)}{\rO(n)\times \rO(n-k, k)}}
\def\oni{\frac{\rO(2n+1-k, k)}{\rO(n+1)\times \rO(n-k,k)}}
\def\onm{\frac{\rO(2n+m-k, k)}{\rO(n+m)\times \rO(n-k, k)}}

\def\ik{isothermic$_k$}
\def\gk{Guichard$_k$}

\def\CT{Combescure transform}

\def\CSGC{Combescure sequence of \gk\ orthogonal co-ordinate systems}
\def\CSGCs{Combescure sequences of \gk\ orthogonal co-ordinate systems}
\def\CSI{Combescure sequence of \ik\ hypersurfaces}
\def\CSIs{Combescure sequences of \ik\ hypersurfaces}
\def\CS{Combescure sequence}
\def\CSs{Combescure sequences}

\begin{document}

\title[Isothermic Hypersurfaces]
{Isothermic Hypersurfaces in $\R^{n+1}$}
\author{Neil Donaldson$^\dag$}\thanks{$^\dag$Research supported in part by NSF Advance Grant\/}
\address{Department of Mathematics\\
University of California at Irvine, Irvine, CA 92697-3875}
\email{ndonalds@math.uci.edu}
\author{Chuu-Lian Terng$^*$}\thanks{$^*$Research supported
in part by NSF Grant DMS-0707132}
\address{Department of Mathematics\\
University of California at Irvine, Irvine, CA 92697-3875}
\email{cterng@math.uci.edu}
\maketitle

\begin{abstract} 
A diagonal metric $\sum_{i=1}^n g_{ii} \D x_i^2$ is termed \emph{\gk} if $\sum_{i=1}^{n-k}g_{ii}-\sum_{i=n-k+1}^n g_{ii} =0$. A hypersurface in $\R^{n+1}$ is \emph{\ik} if it admits line of curvature co-ordinates such that its induced metric is \gk. Isothermic$_1$ surfaces in $\R^3$ are the classical isothermic surfaces in $\R^3$. Both \ik\ hypersurfaces in $\R^{n+1}$ and \gk\ orthogonal co-ordinate systems on $\R^n$ are invariant under conformal transformations. A sequence of $n$ \ik\ hypersurfaces in $\R^{n+1}$ (\gk\ orthogonal co-ordinate systems on $\R^n$ resp.) is called a Combescure sequence if the consecutive hypersurfaces (orthogonal co-ordinate systems resp.) are related by Combescure transformations. We give a correspondence between Combescure sequences of \gk\ orthogonal co-ordinate systems on $\R^n$ and solutions of the $\onk$-system, and a correspondence between Combescure sequences of \ik\ hypersurfaces in $\R^{n+1}$ and solutions of the $\oni$-system, both being integrable systems. Methods from soliton theory can therefore be used to construct Christoffel, Ribaucour, and Lie transforms, and to describe the moduli spaces of these geometric objects and their loop group symmetries.   
\end{abstract}

\section{Introduction}

A parameterised surface $f(x_1, x_2)$ in $\R^3$ is \emph{isothermic} if $(x_1, x_2)$ is a conformal line of curvature co-ordinate system. For example, constant mean curvature surfaces in $\R^3$ are isothermic away from umbilic points. Classical geometers constructed various geometric transforms for these surfaces that gave methods to generate new isothermic surfaces from a given one. The Christoffel transform associates to each isothermic surface $f_1(x_1,x_2)$ a second isothermic surface $f_2(x_1, x_2)$ (a Christoffel dual of $f_1$) such that the principal curvature directions of $f_1$ and $f_2$ are parallel and the map $f_1(x)\mapsto f_2(x)$ is orientation reversing. A Ribaucour transform is a diffeomorphism $\phi:M\to M^*$ between two surfaces in $\R^3$ satisfying: (i) the normal line of $M$ at $p$ intersects the normal line of $M^*$ at $p^*=\phi(p)$ at equal distance for all $p\in M$, (ii) $\phi$ maps principal directions of $M$ to those of $M^*$. Given an isothermic surface in $\R^3$, one can solve a system of compatible ordinary differential equations to construct a two-parameter family of Ribaucour transforms so that the target surfaces are also isothermic; these transforms are known classically as Darboux transforms. These geometric transforms provide a rich class of isothermic surfaces. The connection of isothermic surfaces to soliton theory was first noted in \cite{Cieslinski1995}, that the Gauss--Codazzi equation for isothermic surfaces in $\R^3$ has a Lax pair and is a soliton equation. Techniques from soliton theory have been used to identify classical transformations of isothermic surfaces in \cite{Cieslinski1997,HertrichJeromin1997}.

A notion of isothermic surfaces in $\R^n$ was introduced in \cite{Bruck2002,Burstall2004}: An immersion $f(x_1,x_2)$ in $\R^n$ is \emph{isothermic} if the normal bundle is flat and $(x_1,x_2)$ are conformal line of curvature co-ordinates. It was proved in \cite{Bruck2002,Burstall2004} that the Gauss--Codazzi equation for isothermic surfaces in $\R^n$ is the soliton equation associated to $\frac{\rO(n+1, 1)}{\rO(n)\times \rO(1,1)}$, and analogues of classical transformations were constructed for these higher co-dimension isothermic surfaces.  

Burstall asked in \cite{Burstall2004}: Is there any interesting theory of isothermic submanifolds of $\R^n$ of dimension greater than two? An attempt was made by Tojeiro in \cite{Tojeiro2006}: A parameterised submanifold $f(x)$ in Euclidean space is \emph{$k$-isothermic} if the normal bundle is flat, $x$ is line of curvature co-ordinates, the induced metric is conformal to a Riemannian product, and there are exactly $k$ distributions $E_1, \ldots, E_k$ such that $TM= \oplus_{i=1}^k E_i$ and each $E_i(x)$ is contained in the common eigenspace of the shape operators $\{A_\vv\vert\vv\in \nu(M)_x\}$. However, the class of $k$-isothermic submanifolds is not as rich and does not have all the geometric transforms that isothermic surfaces have. 

In this paper we give a positive answer to Burstall's question. We first define  Guichard diagonal metrics. Fix $1\leq k\leq n-1$. A diagonal metric $\sum_{i=1}^n g_{ii} \D x_i^2$ is \emph{\gk} if 
\[\sum_{i=1}^{n-k}g_{ii}=\sum_{i=n-k+1}^n g_{ii}.\] 
Let $\I_{n-k,k}=\diag(\epsilon_1, \ldots, \epsilon_n)$ with $\epsilon_i= 1$ for $i\leq n-k$ and $\epsilon_i=-1$ for $n-k< i\leq n$, and $(\ ,\ )_k$ the bilinear form on $\C^n$ defined by
\[(\vec x,\vec y)_k=\vec x^T\I_{n-k, k}\vec y.\]
If $\D s^2=\sum_{i=1}^n u_i^2 \D x_i^2$ is a diagonal metric, then $\D s^2$ is \gk\ if and only if $(\vu,\vu)_k=0$, where $\vu=(u_1,\ldots,u_n)^T$.
We call $\vu$ a \emph{metric field} for $\D s^2$.

An orthogonal co-ordinate system $\phi(x)$ on $\R^n$ is \emph{\gk} if $\phi^*(\D s_0^2)$ is \gk, where $\D s_0^2$ is the standard Euclidean metric on $\R^n$. Let $\phi(x),\psi(x)$ be orthogonal co-ordinate systems on $\R^n$. Classically, a map $\phi(x)\mapsto \psi(x)$ is said to be a \emph{Combescure transform} if $\phi_{x_i}(x)$ is parallel to $\psi_{x_i}(x)$ for all $1\leq i\leq n$. 

A hypersurface in $\R^{n+1}$ is \emph{\ik} if it admits line of curvature co-ordinates such that its induced metric is \gk. It follows from the definition that \ik\ hypersurfaces in $\R^{n+1}$ and \gk\ orthogonal co-ordinate systems on $\R^n$ are invariant under conformal transformations. Note that isothermic$_1$ surfaces in $\R^3$ are the classical isothermic surfaces in $\R^3$.  

Let $f_1,f_2$ be two isothermic hypersurfaces in $\R^{n+1}$. The map $f_1(x)\mapsto f_2(x)$ is called a \emph{\CT} if $(f_1)_{x_i}$ is parallel to $(f_2)_{x_i}$ for all $1\leq i\leq n$.  

\begin{definition} \label{fk}
A sequence $\{\phi_1,\ldots,\phi_n\}$ of \gk\ orthogonal co-ordinate systems on $\R^n$ is called a \emph{Combescure sequence} if (i) $\phi_i(x)\mapsto\phi_{i+1}(x)$ is a \CT\ for all $1\leq i\leq n-1$, (ii) $\vu_1,\ldots,\vu_n$ are linearly independent at every point, where $\vu_i=(u_{i1},\ldots,u_{in})^T$ are metric fields for $\phi_i^*(\D s_0^2)$ (i.e., $\phi_i^*(\D s_0^2)=\sum_{j=1}^n u_{ij}^2\D x_j^2$) chosen in such a way that 
\begin{equation}\label{fa}
\frac{(\phi_i)_{x_j}}{u_{ij}} = \frac{(\phi_1)_{x_j}}{u_{1j}},\qquad 1\leq i, j\leq n.
\end{equation}
\end{definition}

Condition (i) implies that given a choice of metric field $\vu_1$ for $\phi_1^*(\D s_0^2)$, there exist unique metric fields $\vu_i$ for $\phi_i^*(\D s_0^2)$ satisfying \eqref{fa}. Note that the condition that $\vu_1,\ldots,\vu_n$ are independent is independent of choices of $\vu_1$. 

\begin{definition}\label{fb}
A sequence of isothermic hypersurfaces $\{f_1,\ldots,f_n\}$ in $\R^{n+1}$ is called a \emph{Combescure sequence} if (i) $f_i(x)\mapsto f_{i+1}(x)$ is a \CT\ for all $1\leq i\leq n-1$, (ii) $\vu_1,\ldots,\vu_n$ are linearly independent at every point, where $\vu_i=(u_{i1},\ldots,u_{in})^T$ are chosen such that the induced metrics $\I_{f_i}=\sum_{j=1}^n u_{ij}^2\D x_j^2$ satisfy $\frac{(f_i)_{x_j}}{u_{ij}}= \frac{(f_1)_{x_j}}{u_{1j}}$ for all $1\leq i,j\leq n$.
\end{definition}

Next we explain the relation between Combescure sequences and soliton equations. To each symmetric space $U/K$ is associated a soliton hierarchy. The flows in the hierarchy are parameterised by $(\alpha,j)$ with $1\leq\alpha\leq\rk(U/K)$ (the rank of $U/K$) and $j$ a positive integer. Each flow is equivalent to the condition that a family of $\fu^\C$-valued connection 1-forms $\{\theta_\lambda\vert\lambda\in\C\}$ is flat: this family is the \emph{Lax pair} of the flow. In fact, the Lax pair $\theta_\lambda$ of the $(\alpha,j)$-th flow is a degree $j$ polynomial in $\lambda$. For example, the $\rSU(2)$, $\frac{\rSU(2)}{\rSO(2)}$- and $\frac{\rSU(3)}{\rSO(3)}$ hierarchies are hierarchies for the non-linear Schr{\"o}dinger equation, modified KdV equation, and the reduced 3-wave equation respectively. We put all $(\alpha,1)$-th flows with $1\leq\alpha\leq\rk(U/K)$ together to construct a first order non-linear system, the so-called $U/K$-system (cf. \cite{Terng1997}). These integrable systems often arise in submanifold geometry. 

We call a basis $\vC=\{\vc_1,\ldots,\vc_n\}$ of $\R^{n-k,k}$ a \emph{null basis} if each $\vc_i$ is a null vector. One main result of this paper is to give a construction of a \CSI\ $\{f_1^{\xi,\vC},\ldots,f_n^{\xi,\vC}\}$ from a solution $\xi$ of the $\oni$-system and a null basis $\vC$ of $\R^{n-k,k}$. Moreover, we show that:
\begin{enumerate}
\item All \CSIs\ arise from this construction.
\item If $\vec B,\vC$ are two null bases for $\R^{n-k, k}$, then $f_i^{\xi,\vC}(x)\to f_i^{\xi,\vec B}(x)$ is again a \CT\ for all $1\leq i\leq n$.
\item Let $f_1,\ldots,f_{n-1}$ be \ik\ hypersurfaces in $\R^{n+1}$ such that (i) $f_i(x)\mapsto f_{i+1}(x)$ is Combescure for all $1\leq i\leq n-2$, (ii) $\vu_1,\ldots,\vu_{n-1}$ are linearly independent at every point, where $\vu_\ell=(u_{\ell 1},\ldots,u_{\ell n})$ is the metric field for the induced metric $\I_{f_\ell}$ chosen so that $\frac{(f_\ell)_{x_j}}{u_{\ell j}}=\frac{(f_1)_{x_j}}{u_{1j}}$ for all $2\leq \ell \leq n-1$. Then there exists an \ik\ hypersurface $f_n$ so that $\{f_1,\ldots,f_n\}$ is a \CSI\ in $\R^{n+1}$.
\item If $(f_1,f_2)$ is a Combescure sequence of isothermic$_1$ surfaces in $\R^3$, then either $(f_1,f_2)$ is a classical Christoffel pair, or $(f_1,g_2)$ is, where $g_2(x_1,x_2)=f(x_1,-x_2)$ is oppositely oriented to $f_1$.
\end{enumerate}
We have similar results for \CSGCs\ and solutions of the $\onk$-system.  

We apply techniques from soliton theory to construct loop group actions, Ribaucour transforms, and Lie transforms for Combescure sequences of \ik\ hypersurfaces. We prove that Christoffel and Ribaucour transforms commute, and that the conjugation of a Ribaucour transform and a Lie transform is again Ribaucour.  Moreover, our results generalise those for isothermic surfaces.

Results in \cite{Terng1997} and \cite{Terng2005} imply that Combescure sequences of \ik\ hypersurfaces are determined by $n^2$ functions of one variable and a null basis of $\R^{n-k,k}$.

An $n$-dimensional submanifold in $\R^{n+m}$ is \emph{\ik} if it has flat normal bundle and line of curvature co-ordinates such that the first fundamental form is \gk. Isothermic$_k$ submanifolds in $\R^{n+m}$ are invariant under conformal transformations. All results for the co-dimension one case generalise easily to higher co-dimension.

All \gk\ co-ordinate systems and \ik\ hypersurfaces in this paper are defined on simply connected open subsets of $\R^n$. Although the Cauchy problems for the $\onk$- and $\oni$- systems with rapidly decaying initial data on a non-characteristic line have global solutions, the corresponding co-ordinate systems and hypersurfaces may have singularities.  
 
\medskip
The paper is organised as follows: We write down the $\onk$- and $\frac{\rO(2n+m-k, k)}{\rO(n+1)\times \rO(n+m-k, k)}$-systems and their Lax pairs in section 2, give the correspondence between solutions of the $\onk$-system and \CSs\ of \gk\  orthogonal co-ordinate systems on $\R^n$ in section 3. In section 4 we describe the correspondence between solutions of the $\oni$-system and \CSs\ of \ik\ hypersurfaces in $\R^{n+1}$, and construct Christoffel transforms of these. In the final section we (a) use the dressing action of a rational loop with two poles on the space of solutions of integrable systems to construct geometric Ribaucour transforms of \CSs\ of \ik\ hypersurfaces and prove that Christoffel and Ribaucour transforms commute, and (b) construct Lie transforms of \CSs\ of \ik\ hypersurfaces and show that the conjugation of a Ribaucour transform by a Lie transform is again a Ribaucour transform. These geometric transforms can easily be generalised to \CSs\ of \gk\ orthogonal co-ordinate systems on $\R^n$ and \CSs\ of \ik\ submanifolds in Euclidean space. 
 
 \section{The \texorpdfstring{$\onk$- and $\oni$-}{U/K-}systems}\label{as}
 
Let $U$ be a real simple Lie group, $\sigma$ an involution of $U$ (we also use $\sigma$ to denote $\D \sigma_e$ on the Lie algebra $\fu$ of $U$), $K$ the fixed point set of $\sigma$, and $\fk,\fp$ the $\pm 1$ eigenspaces of $\sigma$ on $\fu$. Let $\fa$ be a maximal abelian\footnote{Throughout this paper maximal abelian will mean maximal \emph{semisimple} abelian: when $U/K$ is non-Riemannian there will generally be non-semisimple abelian subalgebras of dimension greater than the rank of $U/K$. E.g. $\fp\cap\mathrm{stab}(\ell)$ where $\ell$ is an isotropic line in $\R^{1,1}$ is 3-dimensional abelian for the rank 2 space $\frac{\rO(4,1)}{\rO(3)\times\rO(1,1)}$.} subalgebra in $\fp$, $\{a_1, \ldots, a_n\}$ a basis of $\fa$, and $\fa^\perp$ the orthogonal complement of $\fa$ with respect to the Killing form. The $U/K$-system \cite{Terng1997} is the following PDE for $v:\R^n\to \fa^\perp\cap\fp$,
 \[[a_i, v_{x_j}]-[a_j, v_{x_i}]-[[a_i, v], [a_j, v]]=0, \quad i\not=j.\]
Set 
\begin{equation}\label{bn}
\theta_\lambda= \sum_{i=1}^n (a_i \lambda + [a_i, v])\D x_i.
\end{equation}
The following Proposition is well-known and the proof follows easily from a direct computation:
 
 \begin{proposition}\label{cc}
  The following statements are equivalent for a smooth map $v:\R^n\to \fa^\perp\cap \fp$ and $\theta_\lambda$ defined by \eqref{bn}:
 \begin{enumerate}
 \item $v$ is a solution of the $U/K$-system,
 \item $\theta_\lambda$ is flat for all $\lambda\in\C$,
 \item $\theta_r$ is flat for some $r\in \R\cup i\R$.
 \end{enumerate}
 \end{proposition}
 
We call $\theta_\lambda$ the \emph{Lax pair} for the solution $v$ of the $U/K$-system. Note that $\theta_\lambda$ satisfies the {\it $U/K$-reality condition\/}:
\[\overline{\theta_{\bar\lambda}}= \theta_\lambda, \quad \sigma(\theta_{-\lambda})= \theta_\lambda.\]

\begin{definition}
Given a solution $v$ of the $U/K$-system, an \emph{extended frame} $E_\lambda$ for $v$ is a parallel frame of $\theta_\lambda$ (i.e., $E_\lambda^{-1}\D E_\lambda= \theta_\lambda$)  that satisfies the $U/K$-reality condition: 
   \begin{equation} \label{bg}
   \quad \overline{E_{\bar\lambda}}=E_\lambda, \quad \sigma(E_{-\lambda})= E_\lambda.
   \end{equation} 
An extended frame $E$ is called the \emph{normalized extended frame for the solution $v$} if $E(0,\lambda)=\I$.
\end{definition}
 
The following Proposition is well-known (cf. \cite{Bruck2002,Burstall2004}).
 
\begin{proposition}\label{ae} Let $v$ be a solution of the $U/K$-system, $\theta_\lambda= \sum_{i=1}^n (a_i\lambda + [a_i, v])\, \D x_i$ the Lax pair of $v$, and $E_\lambda$ an extended frame. Then:
 \begin{enumerate}
 \item[(1)] $E_0\in K$,
 \item[(2)] The gauge transformation
 \[E_0\ast \theta_\lambda = E_0 \theta_\lambda E_0^{-1} - \D E_0 E_0^{-1}=\lambda \sum_{i=1}^n E_0a_i E_0^{-1}\, \D x_i\]
   is a $\fp$-valued closed $1$-form,
 \item[(3)] $Z:=\frac{\partial E}{\partial\lambda} E^{-1}\, \big|_{\lambda=0}$ is a $\fp$-valued map and 
 $\D Z= \sum_{i=1}^n E_0a_i E_0^{-1}\, \D x_i$. 
 \end{enumerate}
 \end{proposition}
 
\begin{proof}
A simple computation implies that $E_0\ast \theta_\lambda = \lambda \Omega$. Since $\Omega$ is flat and $\Omega\wedge \Omega=0$, $\D\Omega=0$. The third statement follows from a direct computation.\qed
\end{proof}
 
 \begin{example}[The $\onk$-system]\label{ac}
 Here the involution on $\rO(2n-k, k)$ is $\sigma(g)= \I_{n,n} g\I_{n, n}^{-1}$, and
\[\fp=\left\{\begin{pmatrix} 0 & -\xi^TJ\\ \xi & 0\end{pmatrix}\, \bigg| \, \xi \in \calM_{n\times n}  \right\}, \qquad J=\I_{n-k,k},\]
where $\calM_{n\times n}$ is the space of real $n\times n$ matrices. Set 
\[a_i =\begin{pmatrix} 0& -e_{ii} J\\ e_{ii} & 0\end{pmatrix}, \quad 1\leq i\leq n.\]
 Then $\{a_1, \ldots, a_n\}$ is a basis of a maximal abelian subalgebra $\fa$ in $\fp$, and moreover 
\[\fa^\perp\cap \fp=\left\{\begin{pmatrix} 0& -F^T J\\ F& 0\end{pmatrix}\,\bigg| \, F=(f_{ij}), f_{ii}=0 \quad \forall\,\, 1\leq i\leq n  \right\}.\]
 The Lax pair for the $\onk$-system is 
\begin{equation}\label{fm}
\theta_\lambda= \begin{pmatrix} \omega & -\lambda\delta J\\ \lambda \delta & \tau\end{pmatrix}, \quad \omega= -\delta JF+ F^T J\delta, \quad \tau= -\delta F^TJ+ F\delta J,
\end{equation}
where $\delta=\diag(\D x_1,\ldots,\D x_n)$.

The $\onk$-system is given by the condition that $\theta_0$ is flat, i.e.,
\[\D \omega= -\omega\wedge \omega, \quad \D \tau= -\tau\wedge \tau.\]
 In other words, the $\onk$-system for $F=(f_{ij})$ is
\[\begin{cases} \epsilon_i (f_{ij})_{x_j} + \epsilon_j (f_{ji})_{x_i} -\epsilon_i\epsilon_j \sum_k f_{ik} f_{jk}=0, &i\neq =j,\\
 (f_{ij})_{x_i} + (f_{ji})_{x_j} -\sum_k \epsilon_k f_{ki} f_{kj}=0, & i\neq j,\\
 (f_{ij})_{x_k} + \epsilon_k f_{ik} f_{kj}=0, & i, j, k \quad \text{ distinct}.
 \end{cases}\]
 \end{example}
 
 \begin{example}[The $\oni$-system]\label{ad}
 Here $\sigma(g)= \I_{n+1, n}g\I_{n+1,n}^{-1}$, 
 \[\fp=\left\{\begin{pmatrix} 0& -\xi^T J\\ \xi & 0\end{pmatrix}\, \bigg|\, \xi\in \calM_{n, n+1} \right\},\]
 where $J=\I_{n-k,k}$. Set 
\[a_i= \begin{pmatrix} 0& -A_i^T J\\ A_i & 0\end{pmatrix}, \quad A_i=(e_{ii}, 0), \quad 1\leq i\leq n.\]
 Then $a_1, \ldots, a_n$ forms a basis of a maximal abelian subalgebra $\fa$ in $\fp$, and 
 \[\fa^\perp\cap \fp=\left\{ \begin{pmatrix} 0 & -\xi^TJ\\ \xi & 0\end{pmatrix}\bigg|\, \xi=(F, \gamma), F=(f_{ij}), f_{ii}=0, \,\, \forall \,\, 1\leq i\leq n\right\}.\]
 The $\oni$-system is the PDE for $\xi= (F, \gamma)$ with $F=(f_{ij})$ and $\gamma=(\gamma_1, \ldots, \gamma_n)^T$ such that $f_{ii}=0$ for all $1\leq i\leq n$.
 The Lax pair for the $\oni$-system is
\begin{equation}\label{cf}
 \theta_\lambda= \begin{pmatrix} \omega& -\lambda \begin{pmatrix} \delta\\ 0\end{pmatrix}\, J\\ \lambda (\delta, 0) & \tau \end{pmatrix}, 
 \end{equation}
where $J= \I_{n-k,k}$, $\delta=\diag(\D x_1, \ldots, \D x_n)$, and 
\begin{equation}\label{ap}
\omega= \begin{pmatrix} -\delta JF+ F^T J\delta& -\delta J\gamma\\ \gamma^TJ\delta &0 \end{pmatrix},\qquad \tau= (F\delta- \delta F^T)J, 
\end{equation}
Note that $\omega$ and $\tau$ are $\fo(n+1)$-valued and $\fo(n-k,k)$-valued $1$-forms respectively, and that the $\oni$-system is given by the flatness of $\omega$ and $\tau$. In other words, the $\oni$-system is the following PDE for $F=(f_{ij})$ and $\gamma= (\gamma_1, \ldots, \gamma_n)^T$:
 \begin{equation}\label{aj}
 \begin{cases} \epsilon_i (f_{ij})_{x_j} + \epsilon_j (f_{ji})_{x_i} -\epsilon_i\epsilon_j \sum_k f_{ik} f_{jk}=0, &i\neq j,\\
 (f_{ij})_{x_i} + (f_{ji})_{x_j} -\sum_k \epsilon_k f_{ki} f_{kj}=0, & i\neq j,\\
 (f_{ij})_{x_k} + \epsilon_k f_{ik} f_{kj}=0, & i, j, k \quad \text{ distinct},\\
 (\gamma_i)_{x_j}+ \epsilon_j f_{ij} \gamma_j=0,& i\neq j.
 \end{cases}
 \end{equation}
 \end{example}
  
 \begin{example}[The $\onm$-system]
 The $\onm$-system is the PDE for $\xi=(F, \gamma)$ defined by the flatness of $\omega$ and $\tau$, where $\omega$ and $\tau$ are defined in terms of $F, \gamma$ the same way as in \eqref{ap},  $F=(f_{ij})$ is  an $n\times n$ map with $f_{ii}=0$ for all $1\leq i\leq n$, and $\gamma$ a $n\times m$-valued map.  The Lax pair has the same form as \eqref{cf}.
 \end{example}

 \section{Guichard orthogonal co-ordinate systems}\label{at}
 
We exhibit a relation between \CSs\ of \gk\ orthogonal co-ordinate systems on $\R^n$ and solutions of the $\onk$-system.  

First recall a simple and well-known Lemma (the proof follows from a direct computation): 

\begin{lemma}\label{da}
Let $\phi(x)$ be an orthogonal system on $\R^n$, and $\phi^*(\D s_0^2)=\sum_{i=1}^n u_i^2\D x_i^2$, where $\D s_0^2$ is the standard Euclidean metric on $\R^n$. Set 
\[\ve_i=\frac{\phi_{x_i}}{u_i},\quad g=(\ve_1,\ldots,\ve_n),\quad \omega_i=u_i\D x_i,\]
and $F=(f_{ij})$ with
\[f_{ij}=
\begin{cases}
-\epsilon_i\frac{(u_i)_{x_j}}{u_j}, & i\not=j,\\
0, & i=j. 
\end{cases}\]
Then $\omega:=(\omega_{ij})= g^{-1}\D g$ is the flat Levi-Civita connection 1-form for $\phi^*(\D s_0^2)$ and
\[\omega_{ij}=\frac{(u_i)_{x_j}}{u_j} \D x_i - \frac{(u_j)_{x_i}}{u_i}\D x_j =\epsilon_i f_{ij} \D x_i + \epsilon_jf_{ji}\D x_j,\] or equivalently, 
$\omega=(\omega_{ij})= -\delta JF+ F^T J\delta$, where $\delta= \diag(\D x_1,\ldots,\D x_n)$.
\end{lemma}

\begin{theorem}\label{ak}
Let $F=(f_{ij})$ be a solution of the $\onk$-system, and $E_\lambda$ an extended frame of $F$. Then:
\begin{enumerate}
\item[(a)] $E_0=\begin{pmatrix} g_1&0\\ 0 & g_2\end{pmatrix} \in \rO(n)\times \rO(n-k,k)$.  
\item[(b)] $\frac{\partial E}{\partial \lambda} E^{-1}\,\big|_{\lambda=0}$ is of the form $\begin{pmatrix} 0& Y\\ -JY^T &0\end{pmatrix}$ for some $\calM_{n\times n}$-valued map $Y$ and $\D Y= -\sum_{i=0}^n g_1 e_{ii}Jg_2^{-1}\D x_i$.
\item[(c)] Let $\vc\in \R^{n-k,k}$ be a constant null vector. Then $Y\vc$ is a \gk\ orthogonal co-ordinate system on $\R^n$ with $\D s^2=\sum_{i=1}^n u_i^2 \D x_i^2$ defined on the open subset $\{x\in \R^n\vert \prod_{i=1}^n u_i(x)\not=0\}$, where  $(u_1,\ldots,u_n)^T:= Jg_2^{-1}\vc$.
\item[(d)] If $\vC=\{\vc_1,\ldots,\vc_n\}$ is a null basis for $\R^{n-k,k}$, then $Y\vC=\{Y\vc_1,\ldots,Y\vc_n\}$ is a \CSGC\ on $\R^n$.
\item[(e)] If $\{\vb_1,\ldots,\vb_n\}$ is another null basis for $\R^{n-k,k}$, then $Y(x)\vc_i\mapsto Y(x)\vb_j$ is a Combescure transform for all $1\leq i,j\leq n$. 
\end{enumerate}
\end{theorem}

\begin{proof}
(a) Since $E_\lambda$ is an extended frame, it satisfies the $U/K$-reality condition \eqref{bg},  $E_0\in K$. Hence we can write $E_0=\begin{pmatrix} g_1& 0\\ 0& g_2\end{pmatrix}$. 

(b) By Proposition \ref{ae}, we have $\D Y=-g_1\delta J g_2^{-1}$. Set $g_1=(\ve_1,\ldots,\ve_n)$ and $\vu=(u_1,\ldots,u_n)^T:= g_2^{-1}\vc$. Then 
\begin{equation}\label{bm}
\D Y{\vc}=-(\ve_1,\ldots,\ve_n)\delta(\epsilon_1u_1,\ldots,\epsilon_nu_n)^T=-\sum_{i=1}^n\epsilon_iu_i\D x_i\ve_i.
\end{equation}

(c) The Euclidean metric is therefore $\D s^2=\sum_{i=1}^n u_i^2 \D x_i^2$. But $\vc$ being a null vector in $\R^{n-k, k}$ implies that $(\vu,\vu)_k=0$. Hence $Y{\vc}$ is a \gk\ orthogonal co-ordinate system. 

Statements (d) and (e) follow.\qed
\end{proof}

\begin{theorem}\label{fc}
Let  $\{\phi_1, \ldots, \phi_n\}$ be a \CSGC\ on $\R^n$, and $\phi_i(0)= 0$. Then there is a solution $F$ of the $\onk$-system and a null basis $\vC=\{\vc_1,\ldots,\vc_n\}$ of $\R^{n-k,k}$ so that $\{\phi_1,\ldots,\phi_n\}$ is the \CSGC\ constructed from $F$ and $\vC$ as in Theorem \ref{ak}.
\end{theorem}

\begin{proof}
Let $\vu_\ell=(u_{\ell 1},\ldots,u_{\ell n})^T$ be the metric field for $\phi_\ell^*(\D s_0^2)$ chosen as in Definition \ref{fk}, so $\frac{(\phi_\ell)_{x_j}}{u_{\ell j}}=\frac{(\phi_1)_{x_j}}{u_{1j}}$. Then the vector field $\ve_j:=\frac{(\phi_i)_{x_j}}{u_{ij}}$ is independent of $i$. Set 
\[h=(\ve_1,\ldots,\ve_n),\quad \omega=h^{-1}\D h.\]
Since $h$ does not depend on $\ell$, we have, by Lemma \ref{da}, that 
\[f_{ij}:=-\epsilon_i\frac{(u_{\ell i})_{x_j}}{u_{\ell j}}\]
is independent of $\ell$ for $i\neq j$. Let $F=(f_{ij})$ with $f_{ii}=0$ for all $i$. Then 
\[\omega=-\delta JF + F^TJ\delta= h^{-1}\D h\]
is a flat $\fo(n)$-valued connection 1-form, where $\delta=\diag(\D x_1,\ldots,\D x_n)$.

\smallskip
\noindent{\bf Claim 1.} $F$ is a solution of the $\onk$-system.

The Lax pair $\theta_\lambda$ for the $\onk$-system is \eqref{fm}. By Proposition \ref{cc}, $F$ is a solution of the $\onk$-system if and only if both $\omega$ and $\tau$ are flat, where $\omega=-\delta JF + F^TJ\delta$ and $\tau$ is the $\fo(k, n-k)$-valued 1-form defined by 
\[\tau= (\tau_{ij})= (-\delta F^T + F\delta)J, \quad \text{ i.e., } \tau_{ij}= -f_{ji}\epsilon_j \D x_i + f_{ij} \epsilon_j \D x_j.\] 
To prove Claim 1 it thus suffices to prove that $\tau$ is flat. We will show that
\begin{equation}\label{bk}
\D\vu_\ell^T=\vu_\ell^T \tau, \qquad 1\leq \ell \leq n,
\end{equation}
i.e., $\vu_1^T,\ldots,\vu_n^T$ are linearly independent parallel sections for $\tau$. It then follows that $\tau$ is flat. By the definition of $f_{ij}$, 
\[(u_{\ell i})_{x_j}= -\epsilon_i f_{ij} u_{\ell j}, \quad i\neq j.\]
Since $\sum_{i=1}^n \epsilon_i u_{\ell i}^2=0$, 
\[\epsilon_iu_{\ell i}(u_{\ell i})_{x_i}= -\sum_{j\neq i} \epsilon_j u_{\ell j} (u_{\ell j})_{x_i} = u_{\ell i} \sum_{j\neq i} f_{ji} u_{\ell j}.\]
So we have 
\begin{equation} \label{ab} 
(u_{\ell i})_{x_i} = \epsilon_i\sum_{j\neq i} f_{ji} u_{\ell j}.
\end{equation}
It is a short calculation using \eqref{ab} to establish \eqref{bk}. This proves Claim 1. 

\smallskip

Since $\tau^TJ+ J\tau=0$ and \eqref{bk}, we have  
\begin{equation}\label{db}
\D \vu_\ell = \tau^T \vu_\ell = - J\tau J\vu_\ell, \qquad 1\leq \ell\leq n.
\end{equation}

Let  $\theta_\lambda$ denote the Lax pair for $F$, and $E_\lambda(x)$ the solution of
\[E_\lambda^{-1}\D E_\lambda= \theta_\lambda =
\begin{pmatrix} 
\omega & -\lambda\delta J\\
\lambda\delta & \tau
\end{pmatrix}, \qquad 
E_\lambda(0)=\begin{pmatrix}
h(0) & 0\\
0 & \I
\end{pmatrix},\]
where $h(0)=(\ve_1(0),\ldots,\ve_n(0))\in\rO(n)$. Since $E_\lambda(0)\in\rO(n)\times\rO(n-k, k)$ and $\theta_\lambda$ satisfies the $\onk$-reality condition, $E_\lambda$ satisfies the $\onk$-reality condition. In particular, $E_0(x)\in\rO(n)\times\rO(n-k)$ so we may write 
 $E_0=\begin{pmatrix}
g_1& 0\\
0 & g_2
\end{pmatrix}$. 
Then
\[g_1^{-1}\D g_1=\omega,\qquad g_2^{-1}\D g_2= \tau.\]
But $h^{-1}\D h=\omega$ and $g_1(0)= h(0)$. So $g_1= h$.  
 
\smallskip
\noindent {\bf Claim 2.}  $\vu_\ell = -Jg_2^{-1}J \vu_\ell(0)$. 
 
Since $g_2^{-1}\D g_2= \tau$, 
\[\D (Jg_2^{-1})= -J \tau g_2^{-1}= -J\tau J (Jg_2^{-1}).\] 
By \eqref{db}, $h_2=(\vu_1,\ldots,\vu_n)$ also satisfies $\D h_2=-J\tau J h_2$. There is thus a constant null vector $\vc_\ell$ such that $\vu_\ell= - Jg_2^{-1} \vc_\ell$. Since $g_2(0)= \I$, $\vc_\ell = -J\vu_\ell(0)$. This proves Claim 2. 
 
\smallskip
By Theorem \ref{ak},  
\[\left.\frac{\partial E_\lambda}{\partial\lambda}E_\lambda^{-1}\right|_{\lambda=0}=\begin{pmatrix}0 & Y\\ -J Y^T&0  \end{pmatrix},\]
for some $n\times n$ valued map $Y$. Since $E_\lambda(0)$ is independent of $\lambda$, $Y(0)$ is the zero matrix.  
 
\smallskip
It remains to prove that $Y_\ell= \phi_\ell$, where $Y_\ell= Y\vc_\ell$.

Use $\vu_\ell=-Jg_2^{-1}\vc_\ell$, $g_1= h=(\ve_1,\ldots,\ve_n)$, and Theorem \ref{ak} (b) to see that $\D Y_\ell=\sum_i u_{\ell i}\ve_i \D x_i$. But $\D\phi_\ell = \sum_i u_{\ell i}\ve_\ell \D x_i$ and $Y_\ell(0)= \phi_\ell(0)= 0$. Hence $Y_\ell= \phi_\ell$. 
\qed
\end{proof}

Since an $\fo(n-k,k)$-valued connection 1-form $\tau$ with $n-1$ linearly independent parallel sections is flat, we have:

\begin{corollary}
Suppose $\phi_1,\ldots,\phi_{n-1}$ are \gk\ orthogonal coordinate systems on $\R^n$ such that (i) $\phi_i(x)\mapsto \phi_{i+1}(x)$ is a Combescure transform for $1\leq i\leq n-2$, (ii) $\vu_1,\ldots,\vu_{n-1}$ are linearly independent at every point, where $\vu_\ell=(u_{\ell 1},\ldots, u_{\ell n})$ is the metric field of $\phi_i^*(\D s_0^2)$ chosen such that $\frac{(\phi_\ell)_{x_j}}{u_{\ell j}}=\frac{(\phi_1)_{x_j}}{u_{1j}}$ for all $2\leq \ell\leq n-1$ and $1\leq j\leq n$. There then exists a \gk\ orthogonal co-ordinate system $\phi_n$ on $\R^n$ such that $\{\phi_1,\ldots,\phi_n\}$ is a \CSGC.
\end{corollary} 

As a consequence of Theorem \ref{ak}(e) and Theorem \ref{fc} we have:

\begin{corollary}
Let $\Phi=\{\phi_1,\ldots,\phi_n\}$ be a \CS\ of orthogonal co-ordinate systems on $\R^n$, and $\vec B=\{\vb_1,\ldots,\vb_n\}$ a null basis of $\R^{n-k, k}$. $\vec B$ then defines a \CSGC\ $\Psi=\{\psi_1,\ldots,\psi_n\}$ on $\R^n$ such that $\phi_i(x)\mapsto \psi_i(x)$ is a \CT\ for each $i$.
\end{corollary}

\section{Isothermic hypersurfaces\texorpdfstring{ in $\R^{n+1}$}{}}\label{au}

In this section we describe the correspondence between solutions of the $\oni$-system and \ik\ hypersurfaces in $\R^{n+1}$.

The following result can be proved in a similar manner to Theorem \ref{ak}:

\begin{theorem}\label{am}
Let $(F,\gamma)$ be a solution of the $\oni$-system, and $E_\lambda$ an extended frame for $(F,\gamma)$. Then:
\begin{enumerate}
\item[(a)] $E_0\in \rO(n+1)\times \rO(n-k, k)$, so $E_0= \begin{pmatrix} g_1&0\\ 0& g_2\end{pmatrix}$.
\item[(b)] $\frac{\partial E}{\partial \lambda} E^{-1}\, \big|_{\lambda=0}$ is of the form $\begin{pmatrix}
0& Y\\
-JY^T&0
\end{pmatrix}$ for some $(n+1)\times n$ valued map $Y$, and                                                                                     $\D Y= -\sum_{i=1}^n g_1 e_{ii} Jg_2^{-1}\D x_i$, where $J=\I_{n-k,k}$. 
\item[(c)] Let $\vb$ be a constant null vector in $\R^{n-k,k}$, then 
$Y{\vb}$ is an immersed \ik\ hypersurface in $\R^{n+1}$ with
\[\I_{\vb}= \sum_{i=1}^n u_i^2 \D x_i^2, \quad \II_{\vb}= \sum_{i=1}^n \epsilon_i \gamma_iu_i \D x_i^2,\]
where $(u_1, \ldots, u_n)^T= Jg_2^{-1}\vb$.
\item[(d)] If $\vC=\{\vc_1, \ldots, \vc_n\}$ is a null basis of $\R^{n-k,k}$, then $Y\vC= \{Y\vc_1,\ldots,Y\vc_n\}$ is a \CSI\ in $\R^{n+1}$.
\item[(e)] If $\{\vb_1,\ldots,\vb_n\}$ is another null basis of $\R^{n-k,k}$, then $Y\vc_i(x)\mapsto Y\vb_i(x)$ is a \CT\ for $1\leq i\leq n$. 
\end{enumerate}
\end{theorem}

We need the following Lemma to prove an analogue of Theorem \ref{fc} for \CSIs\ in $\R^{n+1}$: 

\begin{lemma}\label{fg}
Suppose that $f(x_1, \ldots, x_n)$ is an immersed \ik\ hypersurface in $\R^{n+1}$ with fundamental forms 
\[\I=\sum_{i=1}^n u_i^2\D x_i^2, \qquad \II= \sum_{i=1}^n u_i h_i \D x_i^2,\]
where $(\vu,\vu)_k=0$ and $\vu=(u_1,\ldots,u_n)$. Set $F=(f_{ij})$ and $\gamma=(\gamma_1, \ldots, \gamma_n)^T$ by
\begin{equation} \label{af}
f_{ij}=\begin{cases}
-\epsilon_i \frac{(u_i)_{x_j}}{u_j},& i\neq j,\\
0,& i=j,
\end{cases}\qquad
\gamma=(\gamma_1,\ldots,\gamma_n)^T,\quad\gamma_i=\epsilon_i h_i,
\end{equation}
where $J=\I_{n-k,k}=\diag(\epsilon_1,\ldots,\epsilon_n)$. Then 
\begin{enumerate}
\item[(a)]
\begin{equation}\label{dd}
\omega= \begin{pmatrix}
-\delta J F+ F^T J\delta & - \delta J \gamma\\
\gamma^T J\delta &0\end{pmatrix}, 
\end{equation} is flat,
\item[(b)] $\D \vu^T= \vu^T\tau$, where 
\begin{equation}\label{de}
\tau = -\delta F^T J+ F \delta J.
\end{equation}
\end{enumerate}
\end{lemma}

\begin{proof}
Set 
\[\ve_i= \frac{f_{x_i}}{u_i},\quad \omega_i= u_i \D x_i,  \quad 1\leq i\leq n,\]
and $\ve_{n+1}$ the unit normal field. Then $\{\omega_1,\ldots,\omega_n\}$ is the dual frame and
\[\D f= \sum_{i=1}^n \omega_i \ve_i =\sum_{i=1}^n u_i \ve_i \D x_i,\]
Write 
$\D \ve_i =\sum_{j=1}^{n+1} \omega_{ji} \ve_j$ for $ 1\leq i\leq n+1$.
Then 
\begin{equation}\label{fh}
\omega_{ij}= (\D \ve_j, \ve_i)_o, \qquad 1\leq i, j\leq n+1.
\end{equation}
It follows from elementary hypersurface theory that $(\omega_{ij})_{1\leq i,j\leq n}$ is the Levi-Civita connection 1-form for the induced metric $\I_f$ and
\begin{equation}\label{fj}
\omega_{ij}= \frac{(u_i)_{x_j}}{u_j} \D x_i - \frac{(u_j)_{x_i}}{u_i} \D x_j, \quad \omega_{n+1, i} = - h_i \D x_i, 
\quad 1\leq i, j\leq n.
\end{equation}
The Gauss--Codazzi equation for \ik\ hypersurfaces in $\R^{n+1}$ is given by the flatness of the $\fo(n+1)$-valued 1-form $\omega=(\omega_{ij})_{1\leq i, j\leq n+1}$, which written in terms of $F,\gamma$ is \eqref{dd}. This proves statement (a).   
 
Use the condition $\sum_{i=1}^n\epsilon_iu_i^2=0$ and the same computation as for the \gk\ orthogonal co-ordinate system case to conclude $\D \vu^T= \vu^T\tau$. 
\qed
\end{proof}

We have an analogue of Theorem \ref{fc} for \CSIs\ in $\R^{n+1}$:

\begin{theorem}\label{fd}
Let $\{f_1,\ldots,f_n\}$ be a \CS\ of \ik\ hypersurfaces in $\R^{n+1}$, and $f_i(0)= 0$. Then there is a solution $(F,\gamma)$ of the $\oni$-system and a null basis $\vC=\{\vc_1,\ldots,\vc_n\}$ so that $\{\phi_1,\ldots,\phi_n\}$ is the \CSI\ of $\R^{n+1}$ constructed from $(F,\gamma)$ and $\vC$ as in Theorem \ref{am}.
\end{theorem}

\begin{proof}
Let $\vu_\ell=(u_{\ell 1},\ldots,u_{\ell n})$ be the metric fields for $\I_{f_\ell}$ (i.e., $\I_{f_\ell}=\sum_{i=1}^n u_{\ell i}^2\D x_i^2$) such that 
\[\frac{(f_\ell)_{x_i}}{u_{\ell i}}=\frac{(f_1)_{x_i}}{u_{1i}}\]
for all $1\leq\ell\leq n$. We denote these vector fields by $\ve_1,\ldots,\ve_n$ (independent of $\ell$). This means that $(\ve_1,\ldots,\ve_n)$ is the orthonormal principal curvature frame for each $f_\ell$. So the unit normal $\ve_{n+1}$ for $f_1$ is also the unit normal for each $f_\ell$.  By \eqref{fh} the flat $\fo(n+1)$ connection 1-form $(\omega_{ij}(\ell))$ 
\[(\omega_{ij})(\ell)=(\D \ve_j, \ve_i)_0,\]
is independent of $\ell$, denoted it by $\omega=(\omega_{ij})$. But $\omega_{ij}$ is also given by \eqref{fj}. So $\frac{(u_{\ell i})_{x_j}}{u_{\ell j}}$ is independent of $\ell$. Let $F_\ell$ and $\gamma_\ell$ denote the matrix maps defined for $f_\ell$ as in \eqref{af}. So we have proved that $F_\ell=F$ and $\gamma_\ell=\gamma$ are independent of $\ell$, and $\omega$ is given by \eqref{dd}. Let $\tau$ be the $\fo(n-k, k)$ connection 1-form defined by \eqref{de}. By Lemma \ref{fg} (b), $\D \vu_\ell^T= \vu_\ell^T \tau$ for $1\leq \ell\leq n$, i.e., $\vu^T_\ell$ is a parallel frame for $\tau$. By Definition \ref{fb} of a \CSI, $\vu_1,\ldots,\vu_n$ are linearly independent. This proves that $\tau$ is flat. In Example \ref{ad}, we see that the Lax pair $\theta_\lambda$ for $(F,\gamma)$ at $\lambda=0$ is $\begin{pmatrix} \omega&0\\0&\tau\end{pmatrix}$. By Proposition \ref{cc} (3), $(F,\gamma)$ is a solution of the $\oni$-system. 

The rest of the proof can be carried out the same way as that of Theorem \ref{fc}. 
\qed
\end{proof}

Since an $\fo(n-k, k)$-connection $\tau$ with $n-1$ linearly independent parallel sections is flat, we have

\begin{corollary}
Suppose $f_1,\ldots,f_{n-1}$ are \ik\ of $\R^{n+1}$ such that (i) $f_i(x)\mapsto f_{i+1}(x)$ is Combescure for each $1\leq i\leq n-2$, (ii) $\vu_1,\ldots,\vu_{n-1}$ are linearly independent at every point, where $\vu_\ell=(u_{\ell 1},\ldots,u_{\ell n})$ is the metric field for the induced metric $\I_{f_\ell}$ chosen in such a way that $\frac{(f_\ell)_{x_j}}{u_{\ell j}}=\frac{(f_1)_{x_j}}{u_{1j}}$ for all $2\leq\ell\leq n-1$ and $1\leq j\leq n$. There then exists an \ik\ hypersurface $f_n$ in $\R^{n+1}$ such that $\{f_1,\ldots,f_n\}$ is a \CSI.  
\end{corollary}

\smallskip
\noindent{\bf Classical Christoffel transforms}\nopagebreak
\smallskip

When $n=2$, we may assume that $\vc=(1,1)^T$, $\vb=(1,-1)^T$, $f=Y\vc$, $g_1=(\ve_1,\ve_2,\ve_3)$, $\I_f= e^{2u}(\D x_1^2+ \D x_2^2)$, and $g_2^{-1}= \begin{pmatrix} \cosh u& \sinh u\\ \sinh u & \cosh u\end{pmatrix}$. Then $f(x)=Y(x)\vc\mapsto \tilde f(x)= Y(x)\vb$ given in Theorem \ref{am} is a classical Christoffel transform of isothermic surfaces in $\R^3$.  Since $g_2^{-1}\vc= e^u(1,1)^T$ and $g_2^{-1}\vb= e^{-u}(1,-1)^T$, we have $\D f= e^u (\ve_1\D x_1+ \ve_2 \D x_2)$, $\D\tilde f= e^{-u}(\ve_1\D x_1 - \ve_2\D x_2)$, $\ve_1,\ve_2$ are principal directions, and $f\mapsto \tilde f$ is orientation reversing. 

\medskip
\noindent {\bf Isothermic$_k$ submanifolds}\nopagebreak
\smallskip

An $n$-dimensional submanifold in $\R^{n+m}$ is \ik\ if it has flat normal bundle and line of curvature co-ordinates so that the induced metric is \gk. Theorems \ref{am} and \ref{fd} for \ik\ hypersurfaces can be generalised easily to \ik\ submanifolds (with essentially the same statements except replacing $\oni$ by $\frac{\rO(2n+m-k, k)}{\rO(n+m)\times \rO(n-k, k)}$). In particular, 2-submanifolds of $\R^{2+m}$ are isothermic$_1$ if and only if they are isothermic in the sense of \cite{Bruck2002} and \cite{Burstall2004}.

\section{Loop Groups, Ribaucour and Lie transformations}

Classically, geometric transforms for isothermic surfaces were constructed using differential geometric techniques. The Gauss--Codazzi equations for surfaces in $\R^3$ admitting a huge class of geometric transforms often turn out to be soliton equations, and their geometric transforms can be constructed from the dressing action on the space of solutions. Although the classical geometric constructions are beautiful, they often seem mysterious. But if the Gauss--Codazzi equation is a soliton equation, then the techniques from soliton theory give a unified method to construct these geometric transforms. In this section, we first review the dressing action of a loop group on the space of solutions of the $U/K$-system, then use the dressing action of a rational loop with two poles and the scaling transformation of solutions of the $U/K$-system to, respectively, construct geometric Ribaucour and Lie transforms for \CSs\ \ik\ hypersurfaces in $\R^{n+1}$. We also show that the conjugation of a Ribaucour transform by a Lie transform is Ribaucour and that Christoffel transforms commute with Ribaucour transforms.   

\smallskip\noindent{\bf The Dressing action}\nopagebreak
\smallskip

Fix $\epsilon>0$. Let $L(U^\C)$ denote the group of holomorphic maps from $\epsilon^{-1} < |\lambda | < \infty$ to $U_{\C}$ satisfying the $U/K$-reality condition \eqref{bg}, $L_+(U^\C)$ the subgroup of $g\in L(U^\C)$ that can be extended holomorphically to $\C$, and $L_-(U^\C)$ the subgroup of $g\in L(U^\C)$ that can be extended holomorphically to $\lambda=\infty$ with $g(\infty)=\I$. The dressing action of $L_-(U^\C)$ on $L_+(U^\C)$ is defined as follows: Given $g_\pm \in L_\pm(U^\C)$, use the Birkhoff Factorization Theorem to see that for generic $g_\pm$ we can factor $g_-g_+ = \hat g_+ \hat g_-$ uniquely with $\hat g_\pm\in L_\pm(U^\C)$, then the dressing action of $L_-(U^\C)$ on $L_+(U^\C)$ is defined by $g_-\ast g_+ = \hat g_+$.  

It is proved in \cite{Terng2000} that the $L_-(U^\C)$ action induces an action on the space of solutions of the $U/K$-system and on the space of normalized extended frames: If $E$ is the normalized frame of a solution $v$ of the $U/K$-system, then $E(x)\in L_+(U^\C)$. Given $g_-\in L_-(U^\C)$, let $\hat E(x)= g_-\ast E(x)$. Recall that $\hat E(x)$ is obtained by factoring $g_-E(x)=\hat E(x) \tilde g_-(x)$ with $\hat E(x)\in L_+(U^\C)$ and $\tilde g_-(x)\in L_-(U^\C)$. Then $\hat E$ is the normalized frame for a new solution $\tilde v= v+ \pi(m_1)$ of the $U/K$-system, where $m_1$ is the coefficient of $\lambda^{-1}$ for $\tilde g_-$, and $\pi(\xi)$ is the projection from $\fp$ to $\fa^\perp\cap \fp$.  Note that Theorem \ref{av} is a special case of this general Theorem when we take $g_-= p_{\alpha, L}$ defined by \eqref{bd}. In this case, the factorization is carried out explicitly.  

\smallskip\noindent{\bf The Cauchy problem}\nopagebreak
\smallskip

The $U/K$-system as an exterior differential system is involutive and, by Cartan--K\"ahler Theory, local analytic solutions are determined by a germ of an $\fa^\perp\cap\fp$-valued analytic function on a regular line (cf. \cite{Terng2005}). Recall that $a\in \fa$ is regular if $\ad(a):\fa^\perp\cap \fp\to \fk$ is injective. Solutions in the orbit of $L_-(U^\C)$ at the vacuum solution are local analytic.

We can also apply the theory of inverse scattering to find global solutions to the Cauchy problem (cf. \cite{Terng1997}). Explicitly: fix a regular $b\in\fa$, and let $\xi_0:\R\to\fa^\perp\cap\fp$ be rapidly decreasing and have $L^1$-norm less than 1; then there exists a unique global solution $\Xi$ to the $U/K$-system such that $\Xi(tb)=\xi_0(t)$.

\smallskip\noindent{\bf The Dressing action of a simple element}\nopagebreak
\smallskip

Sections 3 and 4 explain how to construct \CSs\ of \gk\ orthogonal co-ordinate systems on $\R^n$ and \ik\ hypersurfaces in $\R^{n+1}$ from solutions of the $\onk$-systems and $\oni$-systems respectively. The dressing action of the loop group on the space of solutions of the $U/K$-system gives rise to various transformations of the corresponding geometric objects. However only the action of the simplest kind of rational loops in $L_-(U^\C)$ are known to give geometrically interesting transforms.  

The dressing action of a rational germ on the space of solutions of the $U/K$-system can be computed by residue calculus. We can thus use Theorems \ref{ak} and \ref{am} to construct geometric transformations of \gk\ orthogonal co-ordinate systems on $\R^n$ and \ik\ hypersurfaces in $\R^{n+1}$. We give the explicit formulae for the dressing actions of the simplest type of rational loops and show that the corresponding transformations of \ik\ hypersurfaces are Ribaucour transformations. 

\smallskip\noindent{\bf Simple elements $p_{\alpha, L}$}\nopagebreak
\smallskip

Let $\sigma$ be the conjugation by $\rho=\I_{n+1,n}$ as in Example \ref{ad} that gives the symmetric space $\frac{U}{K}= \frac{O(2n+1-k, k)}{O(n+1)\times O(n-k, k)}$. Choose a scalar $\alpha\in\R^\times\cup i\R^\times$, and an isotropic line $\ell$ such that either
\begin{equation}\label{ba}
\ell\le\R^{2n+1-k, k}\text{ and }\alpha\in\R^\times,\quad\text{or}\quad \ell\le\R^{n+1}\oplus i\R^{n-k,k}\text{ and }\alpha\in i\R^\times,
\end{equation}
where $\R^\times=\R\setminus\{0\}$. Let $L=\ell^\C$ and suppose in addition that $\rho L\neq L$ (equivalently $\rho L\not\perp L$). Let $\pi_L$ denote the projection onto $L$ away from $\rho L^\perp$. In fact, if $\ell=\langle \vv\rangle$ with $(\vv, \vv)_k=0$ and $(\vv,\rho \vv)_k=1$, then 
\[\pi_L= \vv\vv^T\rho, \quad \pi_{\rho L}= \rho \vv\vv^T.\]
Define the \emph{simple element} $p_{\alpha,L}$ by
\begin{equation}\label{bd}
p_{\alpha,L}(\lambda)=\frac{\lambda-\alpha}{\lambda+\alpha}\pi_L+\pi_{(L\oplus\rho L)^\perp}+\frac{\lambda+\alpha}{\lambda-\alpha}\pi_{\rho L}.
\end{equation}
It is easily checked that $p_{\alpha,L}$ satisfies the $U/K$-reality condition. 

The following is known (\cite{Bruck2002,Donaldson2008}) for the $\frac{\rO(2n-1,1)}{\rO(n)\times \rO(n-1,1)}$-system, and exactly the same proof works for the $\frac{\rO(2n+1-k,k)}{\rO(n+1)\times \rO(n-k,k)}$-system and for the $\frac{\rO(2n-k,k)}{\rO(n)\times \rO(n-k,k)}$-system.

\begin{theorem}\label{av}
Let $\frac{U}{K}=\onk$, $E_\lambda(x)$ the normalized extended frame for a solution $(F,\gamma)$ of the $\frac{U}{K}$-system, $\alpha\in \C$, $L=\ell^\C$ an isotropic line in $\C^{2n+1-k, k}$ such that $\ell$ satisfies \eqref{ba}, and $p_{\alpha,L}$ the simple element defined by \eqref{bd}. Then:
\begin{enumerate}  
\item There is an open subset $B$ of the origin in $\R^n$ such that $E_\alpha^{-1}(x)L \neq \rho(E_\alpha^{-1}(x)L)$ for all $x\in B$, where $\rho(g)= \I_{n+1, n}g\I_{n+1,n}^{-1}$,
\item $\tilde L(x)= E_\alpha^{-1}(x)L$ and $\alpha$ satisfy \eqref{ba}, so there exist $\vec Q(x)\in \C^{n+1}$ and $\vec Z(x)\in \C^n$ such that $\tilde L(x)= \left\langle \begin{pmatrix} \vec Q(x)\\ \vec Z(x)\end{pmatrix}\right\rangle$ with $(\vec Q,\vec Q)_0=-(\vec Z,\vec Z)_k= 2$ and $\bar{\vec Q}=\vec Q$ and $\bar{\vec Z}=\sgn(\alpha^2)\vec Z$, 
\item $\tilde E(x): =E(x) p^{-1}_{\alpha,E^{-1}_\alpha(x) L}$
 is an extended frame for a new solution $(\tilde F, \tilde\gamma)$ of the $U/K$-system, where 
 \[(\tilde F, \tilde \gamma)= (F, \gamma)+ \alpha \vec Z\vec Q^T.\]
\end{enumerate}
 \end{theorem}

\smallskip\noindent{\bf Ribaucour transforms}\nopagebreak
\smallskip
 
Let $M$ and $\tilde M$ be two hypersurfaces in $\R^{n+1}$. A diffeomorphism $\phi:M\to \tilde M$ is  a \emph{Ribaucour transform} (cf. \cite{Bruck2002}) if:
\begin{enumerate}
\item[(1)] For each $p\in M$ there is an $n$-sphere $S_{p}$ containing $p$ and $\phi(p)$ such that $M$ and $\tilde M$ are tangent to $S_{p}$  at $p$ and $\phi(p)$ respectively.
\item[(2)] $\D\phi_{p}$ maps eigenvectors of the shape operator of $M$ at $p$ to eigenvectors of the shape operator of $\tilde M$ at $\phi(p)$.
\item[(3)] The tangent line through $p$ in a principal direction $\vv$ meets the tangent line through $\phi(p)$ in the direction of $\D\phi_{p}(\vv)$ at equal distance.
\end{enumerate}

\smallskip\noindent{\it Remarks}
\smallskip
\begin{enumerate}
\item[(i)] The equidistance claim of (3) follows automatically from (1) and (2) provided that the principal directions through $p,\phi(p)$ in fact meet.
\item[(ii)] In classical differential geometry, a {\it sphere congruence\/} in $\R^{n+1}$ is an $n$ parameter family of $n$-spheres ${\mathcal{S}}= \{S_x\, |\, x\in {\mathcal{O}}\}$, where ${\mathcal{O}}$ is an open subset of $\R^n$.  It is easy to see that there are two hypersurfaces $f(x)$ and $\tilde f(x)$ in $\R^{n+1}$ so that $f(x_0), \tilde f(x_0)\in S_{x_0}$ and $S_{x_0}$ is tangent to $f$ and $\tilde f$ at $f(x_0)$ and $\tilde f(x_0)$ respectively.  These are the called {\it envelopes\/} of $\mathcal S$. We also call the resulting map $f(x)\mapsto \tilde f(x)$ a sphere congruence. A Ribaucour transform is thus a sphere congruence that preserves lines of principal curvature.
\item[(iii)]
Let $M$ be a hypersurface in $\R^{n+1}$, and $(\ve_1,\ldots,\ve_{n+1})$ an orthonormal frame on $M$ such that $\ve_1,\ldots,\ve_n$ are principal directions (i.e., unit eigenvectors for the shape operator of $M$). Let $\tilde M$ be another hypersurface, $\phi:M\to \tilde M$  a diffeomorphism, $\tilde\ve_{n+1}$ the unit normal field on $\tilde M$, and $\tilde\ve_i$ is the direction of $\D \phi(\ve_i)$ for $1\leq i\leq n$. Then $\phi$ is a Ribaucour transform if:
\begin{enumerate}
\item[(a)] $\tilde\ve_i$ is a principal direction for $\tilde M$ for $1\leq i\leq n$.
\item[(b)] There exist functions $h_1, \ldots, h_{n+1}$ on $M$ such that
\[\phi(p)+ h_i(p)\tilde\ve_i (p)=p+ h_i (p)\ve_i(p), \qquad 1\leq i\leq n+1\]
for all $p\in M$.   
\end{enumerate}
\end{enumerate}

\begin{theorem} \label{aw}
Let $(F,\gamma)$ be a solution of the $\frac{\rO(2n-1-k, k)}{\rO(n+1)\times \rO(n-k,k)}$-system, $E$ the normalized extended frame for $(F,\gamma)$, $Y$ defined by
 \[\frac{\partial E}{\partial \lambda}E^{-1}\big|_{\lambda=0}=\begin{pmatrix} 0& Y\\ -JY^T&0\end{pmatrix},\]
 where $\vc$ a constant null vector in $\R^{n-k,k}$, and $f_\vc=Y\vc$ (the resulting \ik\ hypersurface in $\R^{n+1}$). Let $p_{\alpha, L}$ be the simple element defined by \eqref{bd}, and $\vec Q, \vec Z$ and $\tilde E$ as in Theorem \ref{av}. Then:
\begin{enumerate}
\item $E_0=\begin{pmatrix}
g_1&0\\
0&g_2
\end{pmatrix}$ and $\tilde E_0=\begin{pmatrix}
\tilde g_1&0\\
0&\tilde g_2
\end{pmatrix}$ for some $g_1,\tilde g_1\in\rO(n+1)$ and $g_2,\tilde g_2\in\rO(n-k,k)$.
\item Write $g_1=(\ve_1,\ldots,\ve_{n+1})$, $\tilde g_1=(\tilde\ve_1,\ldots,\tilde\ve_{n+1})$, $\vu=Jg^{-1}_2\vc=(u_1, \ldots, u_n)^T$, $\vec Q=(q_1,\ldots,q_{n+1})^T$, and $\vec Z=(z_1,\ldots,z_n)^T$. Set 
\[\tilde f_{\vc}:=f_{\vc}-\frac{\vec Z^T \vu}{\alpha}\sum_{j=1}^{n+1}q_j\ve_j.\]
Then 
\begin{enumerate} 
\item $\tilde f_{\vc}$ is \ik,
\item $\{\ve_i\}_{i=1}^n$ and $\{\tilde \ve_i\}_{i=1}^n$ are principal curvature directions of $f_\vc$ and $\tilde f_\vc$ respectively,
\item $R_{\alpha, L}: f_\vc\mapsto \tilde f_\vc$ is a Ribaucour transform, in fact, 
\[\tilde f_\vc-\frac{(\vec Z, \vu)_k}{\alpha q_i}\tilde \ve_i= f_\vc- \frac{(\vec Z, \vu)_k}{\alpha q_i} \ve_i, \qquad 1\leq i\leq n+1.\]
\end{enumerate}
\end{enumerate}
\end{theorem}
 
\begin{proof}
Note that
\begin{equation}\label{az}
p^{-1}_{\alpha,E_\alpha^{-1}L}(\lambda)=\begin{pmatrix}
  \I+\frac{\alpha^2}{\lambda^2-\alpha^2}\vec Q\vec Q^T&-\frac{\alpha\lambda}{\lambda^2-\alpha^2}\vec Q\vec Z^TJ\\
  \frac{\alpha\lambda}{\lambda^2-\alpha^2}\vec Z\vec Q^T&\I-\frac{\alpha^2}{\lambda^2-\alpha^2}\vec Z\vec Z^TJ
\end{pmatrix}, \qquad J= \I_{n-k,k},
\end{equation}
from which we write
\[\tilde E_\lambda = E_\lambda \begin{pmatrix}  \I+\frac{\alpha^2}{\lambda^2-\alpha^2}\vec Q\vec Q^T&-\frac{\alpha\lambda}{\lambda^2-\alpha^2}\vec Q\vec Z^TJ\\
  \frac{\alpha\lambda}{\lambda^2-\alpha^2}\vec Z\vec Q^T&I-\frac{\alpha^2}{\lambda^2-\alpha^2}\vec Q\vec Q^TJ\end{pmatrix}.\]
Evaluate at $\lambda=0$ to get 
 \begin{equation}\label{be}
 \begin{pmatrix} \tilde g_1& 0\\ 0&\tilde g_2\end{pmatrix}=\begin{pmatrix} g_1&0\\ 0& g_2\end{pmatrix}\, \begin{pmatrix} \I-\vec Q\vec Q^T& 0\\ 0 & \I+\vec Z\vec Z^TJ\end{pmatrix}.
 \end{equation}
 Let $\tilde \ve_i$ and $\ve_i$ denote the $i$th column of $g_1$ and $\tilde g_1$ respectively. By \eqref{be} we have
 \begin{gather}
 \tilde \ve_i= \ve_i -q_i \sum_{j=1}^{n+1} q_j \ve_j, \quad 1\leq i\leq n+1. \label{ax2}
 \end{gather}
 Recall that $\vu=Jg_2^{-1}\vc$. Notice that
\[\frac{\partial \tilde E}{\partial \lambda} \tilde E^{-1}\big|_{\lambda=0}=\begin{pmatrix} 0& \tilde Y\\ -J\tilde Y^T&0\end{pmatrix}, \qquad 
  \frac{\partial E}{\partial\lambda} E^{-1}\big|_{\lambda=0}=\begin{pmatrix} 0& Y\\ -JY^T&0\end{pmatrix}, \]
  where $J=\I_{n-k,k}$. By \eqref{az},
\begin{equation}\label{bf}
\tilde Y= Y - \alpha^{-1} g_1 \vec Q\vec Z^T Jg_2^{-1}.
\end{equation}
Set $\tilde f_\vc =\tilde Y\vc$. Then since $f= Y{\vc}$ and \eqref{bf},
\[\tilde f_\vc= f_\vc - \alpha^{-1} g_1 \vec Q\vec Z^T Jg_2^{-1}\vc.\]
Or equivalently, 
\[\tilde f_\vc-f_\vc = -\alpha^{-1}g_1\vec Q (\vec Z, g_2^{-1}\vc)_k =-\frac{\vec Z^T \vu}{\alpha} \sum_{i=1}^{n+1} q_i \ve_i.\]
By \eqref{ax2}, 
\[\tilde \ve_i -\ve_i= -q_i\sum_{j=1}^{n+1} q_j \ve_j, \qquad 1\leq i\leq n+1.\]
Hence 
$\tilde f_\vc + f_i \tilde \ve_i = f_\vc + f_i \ve_i$ for $1\leq i\leq n+1$, where $f_i= -\frac{\vec Z^T \vec u}{\alpha q_i}$.\qed
\end{proof}

\smallskip

As a consequence of Theorems \ref{fd} and \ref{av} we have:

\begin{corollary}
 Let $\{f_1,\ldots,f_n\}$ be  a \CSI\ in $\R^{n+1}$,  and $p_{\alpha,L}$ a simple element defined by \eqref{bd} so that $\alpha, L$ satisfying condition \eqref{ba}. Then there is a new \CS\ of \ik\ hypersurfaces $\{\tilde f_1,\ldots,\tilde f_n\}$ so that $f_i(x)\mapsto\tilde f_i(x)$ is a Ribaucour transform for each $i$.  Moreover, the radius of $S_i(x)$ depends only on $x$ but not on $i$, where $S_i(x_0)$ is the sphere tangent to $f_i$ and $\tilde f_i$ at $f_i(x_0)$ and $\tilde f_i(x_0)$ respectively.
 \end{corollary}

 Ribaucour transforms of a \CS\ of \ik\ hypersurfaces in $\R^{n+1}$ can also be constructed by solving a system of compatible ODEs:

\begin{theorem}\label{cb}
Let $f=\{f_1,\dots,f_n\}$ be a \CS\ of \ik\ hypersurfaces in $\R^{n+1}$ with $\I_{f_1}=\sum_{i=1}^n u_i^2\D x_i^2$, $f_i(0)=0$, and $f_{x_i}(0)$ is parallel to $\vv_i$ for $1\leq i\leq n$, where $\{\vv_1,\ldots,\vv_{n+1}\}$ is the standard basis of $\R^{n+1}$. Let $(F,\gamma)$ denote the solution of the $\oni$-system constructed from $f$ as in Theorem \ref{fd}, i.e.,  $F,\gamma$ are defined by \eqref{af}. Let $\theta_\lambda$ the Lax pair \eqref{cf} for $(F,\gamma)$. Given $\alpha\in \R^{\times}\times i\R^{\times}$ and $\ell=\langle \vv\rangle$ satisfying \eqref{ba}, then the following first order system for $\vec y:\R^n\to \C^{2n}$ is solvable and has a unique solution:
\[\D \vec y= -\theta_\alpha \vec y, \qquad \vec y(0)=\vv.\]
Moreover,
\begin{enumerate}
\item $\langle \vec y(x)\rangle $ satisfies condition \eqref{ba}, so there exists $\vec Q(x)\in \R^{n+1}$ and $\vec Z(x)\in \R^n$ if $\alpha\in \R$ and in $i\R^n$ if $\alpha\in i\R^{\times}$ such that $\langle \vec y \rangle = \left\langle\begin{pmatrix} \vec Q\\ \vec Z\end{pmatrix}\right\rangle$ and $(\vec Q,\vec Q)_0=-(\vec Z,\vec Z)_k= 2$,
\item $\tilde f_i:= f_i- \frac{\vec Z^T \vu}{\alpha} \sum_{j=1}^{n+1} q_j \ve_j$ defines a Ribaucour transform $R_{\alpha, L}$, where $\ve_i= \frac{f_{x_i}}{u_i}$ for $1\leq i\leq n$, $\ve_{n+1}$ is the unit normal for $f$, $\vec Q=(q_1, \ldots, q_{n+1})^T$, and $\vu=(u_1, \ldots, u_n)^T$,
\item $\{\tilde f_1, \ldots, \tilde f_n\}$ is a \CSI.
\end{enumerate}
\end{theorem}

\begin{proof}
Note that Theorem \ref{av}(2) says that $\tilde L= E_\alpha L$, where $E$ is an extended frame of $(F,\gamma)$. Then 
\[\D \tilde L= -E_\alpha^{-1}\D E_\alpha E_\alpha^{-1} L= -\theta_\alpha \tilde L.\]
This gives a system of ODE $\D \vec y= -\theta_\alpha \vec y$.  This system has $n$ independent solutions if and only if $\theta_\alpha$ is flat.  The rest follows.
\end{proof}

Let $R_{\alpha, L}$ and $C_{\vec B}$ denote the Ribaucour and Combescure transforms for \CSs\ of \ik\ hypersurfaces in $\R^{n+1}$. It follows from the constructions of these transforms commute, i.e.,
\[C_{\vec B}\circ R_{\alpha, L}= R_{\alpha, L}\circ C_{\vec B}.\]

\smallskip\noindent{\bf Scaling transforms}\nopagebreak
\smallskip

We review the scaling transform on the space of solutions of the $U/K$-system (cf. \cite{Terng2000}). Let $\Aut(L)$ be the group of automorphisms of the group $L$ defined in the beginning of this section, and $\rho:\R^+\to \Aut(L)$ the group homomorphism defined by 
\[(\rho(r)(g))(\lambda) = g(r\lambda).\]
It is easy to see that $\rho(r)(L_\pm)\subset L_\pm$. The multiplication for the semi-direct product $\R^+\times_\rho L_-$ is
\[(r_1, g_1)(r_2, g_2)= (r_1r_2, g_1 (\rho(r_1) (g_2))).\]
The $\R^+$-action on the solutions of the $U/K$-system is defined as follows: If $r\in \R^+$ and $v$ is a solution of the $U/K$-system, then $\tilde v(x)= r^{-1}v(rx)$ is again a solution and the normalized frames are related by
\[\tilde E(x,\lambda)= E(r^{-1} x, r\lambda).\]
Then $r\ast v= \tilde v$ and $r\ast E= \tilde E$ define an $\R^+$-action on the space of solutions and the space of normalized extended frames of the $U/K$-system.  

It is proved in \cite{Terng2000} that the $\R^+$-action and the $L_-$-action on the spaces of solutions and normalized extended frames extend to an action of the semi-direct product $\R^+\times_\rho L_-$ on these spaces.  Since $(r,1)(1, g)(r^{-1}, 1)= (1, \rho(r)(g))$, we have
\begin{equation} \label{bh}
r\ast (g\ast (r^{-1}\ast v))= (\rho(r)(g))\ast v, \qquad r\ast (g\ast (r^{-1}\ast E))= (\rho(r)(g))\ast E,
\end{equation}
where $v$ is a solution of the $U/K$-system and $E$ is its normalized extended frame. 

\smallskip\noindent{\bf Lie Transforms}\nopagebreak
\smallskip

As a consequence of the $\R^+$-action and Theorem \ref{am} we get,

\begin{theorem}\label{ce}
If $f=\{f_1,\ldots,f_n\}$ is a \CS\ of \ik\ hypersurfaces in $\R^{n+1}$ and $r\in\R^+$, then $\tilde f=\{\tilde f_1,\ldots,\tilde f_n\}$ is a \CSI\ in $\R^{n+1}$, where $\tilde f_i(x)= r f_i(r^{-1}x)$.   We denote this transform by $S_r$. 
\end{theorem}

The scaling transform exists for the Sine-Gordon equation $q_{xt}= \sin q$: $\tilde q(x,t)= q(rx, r^{-1}t)$ is again a solution of the SGE if $q$ is. The Gauss--Codazzi equation for $K=-1$ surfaces in $\R^3$ is the SGE, and the scaling transformation induces a transformation of surfaces in $\R^3$ with $K=-1$. Such transformations were constructed by Lie. We therefore call the transformation given in Theorem \ref{ce} the \emph{Lie transform} for \ik\ hypersurfaces. 

It follows from $\rho(r)(p_{\alpha,L})= p_{r^{-1}\alpha, L}$ and \eqref{bh} that we have:

\begin{corollary}
Let $S_r$ and $R_{\alpha, L}$ denote Lie transform and the Ribaucour transform for \CSIs in $\R^{n+1}$ respectively. Then $S_r\circ R_{\alpha, L} \circ S_{r^{-1}}= R_{r^{-1}\alpha, L}$.   
\end{corollary}

Ribaucour transforms are defined for submanifolds with flat normal bundle (cf. \cite{Bruck2002}): Let $M, \tilde M$ be $n$-dimensional submanifolds in $\R^{n+m}$ with flat normal bundle, and $(\ve_1, \ldots, \ve_n)$ an orthonormal frame on $M$ such that $\ve_1, \ldots, \ve_n$ are common eigenvectors for the shape operators of $M$, and $\ve_{n+1}, \ldots, \ve_{n+m}$ are a parallel normal field to $M$. A Ribaucour transformation from $M$ to $\tilde M$ is a bundle morphism $P:\nu(M)\to \nu(\tilde M)$ which covers a diffeomorphism $\phi:M\to \tilde M$ satisfying the following conditions:
\begin{enumerate}
\item For $1\leq i\leq n$,  $\D \phi_{p}(\ve_i(p))$ is a common eigenvector of the shape operators of $\tilde M$ at $\phi(p)$,   
\item $(\tilde \ve_{n+1}, \ldots, \tilde \ve_{n+m}):=(P(\ve_{n+1}), \ldots, P(\ve_{n+m}))$ is an orthonormal parallel normal frame on $\tilde M$.
\item There exist smooth functions $r_i$ on $M$ such that
\[p+ r_i(p) \ve_i(p)= \phi(p)+ r_i(p) \tilde \ve_i(p), \qquad 1\leq i\leq n+m,\]
where $\tilde \ve_i$ is the direction of $\D \phi (\ve_i)$ for $1\leq i\leq n$.
\end{enumerate}

It follows that if $\phi$ is Ribaucour, then $M$ and $\tilde M$ envelop $m$ congruences of $n$-spheres: fix $n+1\leq i\leq n+m$, the $i$th congruence is the family of $n$-spheres $S_{p}$ in $\R^{n+m}$ with centres $p+r_i(p)\ve_i(p)$ and radii $|r_i(p)|$.

Ribaucour and Lie transforms can similarly be constructed for Combescure sequences of \ik\ $n$-submanifolds in $\R^{n+m}$ as for \CSIs\ in $\R^{n+1}$.\\

The first author was supported in part by an NSF Advance Grant. The second author was supported in part by NSF Grant DMS-0707132.

\end{document}